\documentclass[pdflatex,sn-mathphys-num]{sn-jnl}
\usepackage{hyperref}
\usepackage{amsmath,amsthm,amssymb,bm, comment}
\usepackage{xcolor}
\usepackage{graphicx} 
\usepackage{booktabs}
\usepackage{adjustbox}
\usepackage{multirow}
\usepackage[normalem]{ulem}
\usepackage{anyfontsize}

\usepackage[linesnumbered,ruled,vlined]{algorithm2e}

\DeclareMathOperator{\ord}{ord} \DeclareMathOperator{\lc}{lc}

\newcommand{\cB}{\mathcal{B}}

\newcommand{\cH}{\mathcal{H}}
\newcommand{\cR}{\mathcal{R}}

\newcommand{\cT}{\mathcal{T}}

\newcommand{\bbE}{\mathbb{E}}
\newcommand{\bbN}{\mathbb{N}}
\newcommand{\bbZ}{\mathbb{Z}}

\newcommand{\bbQ}{\mathbb{Q}}

\newcommand{\coC}{\mathbf{C}}
\newcommand\kdv{\text{kdv}}
\newcommand{\coeff}{\operatorname{coeff}}

\newtheorem{thm}{Theorem}[section]
\newtheorem{theorem}[thm]{Theorem}
\newtheorem{prop}[thm]{Proposition}

\newtheorem{lemma}[thm]{Lemma}

\newtheorem{corollary}{Corollary}[thm]

\theoremstyle{definition}

\newtheorem{definition}[thm]{Definition}
\newtheorem{ex}[thm]{Example}

\theoremstyle{remark}
\newtheorem{rem}[thm]{Remark}

\newcommand{\cBa}{W}

\newcommand{\centr}{\mathcal{Z}}

\SetKwInOut{Input}{Input} \SetKwInOut{Output}{Output}
\SetKwComment{Comment}{// }{}

\begin{document}

\title{Computing almost commuting bases of ODOs\\ and Gelfand-Dickey hierarchies}
\author[a]{Antonio Jimenez-Pastor}
\author[a]{Sonia L. Rueda}
\author[b]{Maria-Angeles Zurro}
\author[c]{Rafael Hernandez Heredero}
\author[c]{Rafael Delgado}

\affil[a]{Departamento de Matem\'atica Aplicada, E.T.S.
Arquitectura, Universidad Polit\'ecnica de Madrid}
\affil[b]{Departamento de Matemáticas, Universidad Autónoma de
Madrid} \affil[c]{Departamento de Matemática Aplicada a las
Tecnologías de la Información y las Comunicaciones, ETS de
Ingenieros de Telecomunicación, Universidad Politécnica de Madrid}

\abstract{%
Almost commuting operators were introduced in 1985 by George Wilson
to present generalizations of the Korteweg-de Vries hierarchy,
nowadays known as Gelfand-Dickey (GD) hierarchies. In this paper, we
review the formal construction of the vector space of almost
commuting operators with a given ordinary differential operator
(ODO), with the ultimate goal  of obtaining a basis by computational
routines, using the language of  diffe\-rential polynomials. We use
Wilson's results on weighted ODOs to gua\-ran\-tee the solvability
of the triangular system that allows to compute the homogeneous
almost commuting operator of a given order in the ring of ODOs. As a
consequence, the computation of the equations of the GD hierarchies
is achieved without using pseudo-differential operators. A new
package in SageMath called  \texttt{dalgebra} has been designed to
perform symbolic calculations in differential domains. The
algorithms to calculate the almost commuting basis and the GD
hierarchies in the ring of ODOs are implemented in SageMath, and
explicit examples are provided. }

\pacs[MSC{[}2020{]}]{13N10, 13P15, 12H05}

\maketitle

\section{Introduction}\label{sec:introduction}

Almost commuting operators were introduced in 1985 by George Wilson
in~\cite{Wilson1985} to present the Lax representations and
generalizations of the Korteweg-de Vries hierarchy, nowadays known
as the Gelfand-Dickey (GD) hierarchies \cite{Dikii,
DrinfeldSokolov1985}. The generalization of the Korteweg-de Vries
hierarchy, passing from a second-order linear differential operator
to an $n$-th order operator $L$ was suggested by Kri\-che\-ver, and
implemented by Gelfand and  Dickey, see references in  \cite{Dikii,
DrinfeldSokolov1985}. The standard construction of these generalized
hierarchies is based on fractional powers $L^{m/n}$ of a
differential operator $L$ of order $n$, which live in a ring of
pseudo-differential operators.

The question addressed in this paper is the design of a symbolic
algorithm that allows the computation of the GD hierarchies but does
not require computations with pseudo-differential operators. We
develop an algorithm that requires computations in a ring of
differential operators $\mathcal{D}$, whose coefficients belong to a
ring of differential polynomials. This is made possible by the
notion of \emph{almost commuting operator} defined by G.~Wilson in
~\cite{Wilson1985}, as an operator $A\in \mathcal{D}$ such that
$[L,A]=LA-{ AL}$ has order $\leq n-2$. We reorganized previously
known ideas (by Wilson and many others) to prove that the vector
space $W(L)$, of almost commuting operators with a given ordinary
differential operator $L$
\[W(L)=\{A\in \mathcal{D}\mid \ord [L,A]\leq n-2\}={ Z((L))_+,}\]
equals the set $Z((L))_+$, of positive parts of the elements of the
centralizer $Z((L))$ of $L$ in the ring of pseudo-differential
operators, see Theorem \ref{thm-pcen}.

Our motivation comes from the need of having a basis of almost
commuting operators, and explicit formulas for integrable
hierarchies, to search for new families of algebro-geometric
differential operators, those having centralizers that are not
polynomial rings and containing pairs of ODOs of coprime orders
\cite{We, MRZ2, RZ2024}. Recursive methods exist  to compute almost
commuting bases, for specific values of the order $n$ of $L$. For
$n=2$, recursive formulas for the Korteweg-de Vries hierarchy were
given in  \cite{GH}, that can be obtained by means of a recursion
operator \cite{Olver77}. For $n=3$ recursive formulas can be found
in \cite{DGU}, related with Boussinesq systems. Our ultimate goal is
to compute almost commuting bases by computational routines for an
arbitrary value of $n$. We saw in the work of Grünbaum~\cite{Grun}
the possibility of developing symbolic algorithms for this purpose.
We used preliminary implementations of these algorithms in Maple to
compute examples of algebro-geometric differential operators in
previous works of some of the authors: see \cite{Previato2019,
RZ2021}, and \cite{RZ2024}. New examples of prime order ODOs will
allow one to extend the results in \cite{RZ2024}, and then develop a
suitable Picard-Vessiot theory for spectral problems. In addition,
an operator $A$ of the almost commuting basis of $L$ provides a Lax
representation $L_t=[A,L]$ and is a preparatory step towards the
computation of recursion operators as in~\cite{GKS}.

\medskip

Considering a set $U=\{u_2,\ldots ,u_n\}$ of differential variables
over the field of rational numbers $\bbQ$, with respect to a given
derivation $\partial$, we aim to compute a basis of $W(L)$ for the
formal differential operator
\begin{equation}\label{equ:L}
    L = \partial^n + u_{2}\partial^{n-2} + \ldots + u_{n-1}\partial + u_{n}.
\end{equation}
Wilson defined a weight function on the ring of differential
polynomials $\cR=\bbQ\{U\}$, that guaranteed the existence of a
unique basis $\cH(L)=\{P_m\mid m\in\bbN\}$ { of $W(L)$,} whose
operators are homogeneous with respect to the weight function. A
natural approach to seek for an homogeneous almost commuting
operator  $P_m$ of order~$m$ is to define the formal operator
\[\tilde{P}_m=\partial^m + y_{2}\partial^{m-2} + \ldots + y_{m-1}\partial + y_{m},\]
whose coefficients  are differential variables in $Y=\{y_2,\ldots ,y_m\}$ with respect to~$\partial$, and force $[L,\tilde{P}_m]$ to be a differential operator of order $\leq n-2$. This produces a differential system $\bbE_{m,n}$ of $m-2$ equations given by 
differential polynomials
in $\bbQ\{U,Y\}$. We proved that $\bbE_{m,n}$ is a triangular system
w.r.t.~the variables $Y$ in Theorem \ref{thm-triangularE} and
moreover that
\[\bbE_{m,n}=\{y'_{i}+e_i(Y)=0\mid i=2,\ldots m\}\]
where $e_2\in \bbQ\{U\}$ and
$$e_i(Y)\in \bbQ\{U\}\{y_{2},\ldots ,y_{i-1}\},\,\,\, i\geq 3$$
are differential polynomials, homogeneous with respect to an
extended weight function on $\bbQ\{U,Y\}$. We use Wilson's results
on weighted ODOs to guarantee the solvability of this triangular
system, that allows to compute a unique solution set $Z$ of
$\bbE_{m,n}$ w.r.t.~$Y$. We can solve $\bbE_{m,n}$ by using the
integration methods for differential polynomials described in
~\cite{Bilge1992} and ~\cite{Boulier2016}, combined with backwards
substitution.

Replacing each variable in $Y$ by the corresponding solution in $Z$,
the almost commuting operator $P_m$ is obtained together with the
commutator
\[
{[P_m ,L]}=H_{m,0}+H_{m,1}\partial+\cdots
+H_{m,n-2}\partial^{n-2},\] where $H_{m,k}(U)$, $k=0,\ldots ,n-2$
are the differential polynomials in $\bbQ\{U\}$ that determine the
GD hierarchies {as being defined in \cite{Dikii,
DrinfeldSokolov1985, Wilson1985}}.

A new package in SageMath called  \texttt{dalgebra} has been designed to perform symbolic calculations in differential domains. The algorithms to compute the almost commuting basis and the GD hierarchies in the ring of ODOs, are implemented in the the module \texttt{almost\_commuting.py} ~\cite{JRZ2023,SAGE2023}. 
More details on this implementation and the computations performed
are included in Section~\ref{ExperimentalResults}. A data set has
been generated including almost commuting bases and GD hierarchies
for orders $n=2,3,5,7$ of $L$, considering it would be useful to
researchers in Mathematics and Physics.

\medskip

{\it The paper is organised as follows}. In
Section~\ref{sec:preliminaries} we provide  definitions for the
differential algebra concepts that are used throughout the paper. We
review in Section~\ref{sec:wilson} the main structural result of
Wilson for almost commuting operators. In Section~\ref{sec:almost}
we define the system $\bbE_{m,n}$ that allows the  computation of
almost commuting bases. We will prove that this system is a
triangular differential system with a unique weighted solution $Z$.
We proceed to show in Section~\ref{sec:hierarchies} how the GD
hierarchies are obtained as a by-product, and compute them for an
operator $L$ of orders $3$ and $5$. Finally,
Section~\ref{ExperimentalResults} contains a description of the
implementation in SageMath \cite{SAGE2023} of the symbolic
algorithms to compute almost commuting bases and GD hierarchies.

\section{Preliminaries}\label{sec:preliminaries}

A differential ring $(R, \partial )$ is a ring $R$ (assumed
associative, with identity) equipped with a specified derivation
$\partial$, that is, an additive map $\partial: R\rightarrow R$
satisfying the product rule: $\partial(xy) = \partial(x)y +
x\partial(y)$ for all $x, y\in R$. The elements $r\in R$ such that
$\partial(r) = 0$ are called constants of $R$, and the collection of
these constants forms a subring of $R$. A differential field $(K,
\partial )$ is a differential ring that is a field.

For concepts in differential algebra, related with differential
polynomials, we refer to \cite{Kolchin, Ritt}. Let $\bbN$ be the set
of positive integers including zero. Given a differential variable
$y$ with respect to $R$, we define $y^{(i)}$ as $\partial^i (y)$,
for $y\in\bbN$, $i\geq 1$ and $y^{(0)}$ as $y$. The ring of
differential polynomials over $R$ in the differential variable $y$
is the polynomial ring
\[R\{y\} =  R[y^{(i)}\mid i \in \bbN].\]
The notation $y'$, $y''$, $y'''$ will also be used for $y^{(i)}$,
$i=1,2,3$. Observe that~$R\{y\}$ is a differential ring with the
natural extended derivation, that we also call~$\partial$ abusing
notation. We can iterate this construction, adding several
differential variables. Given differential variables $y$ and $z$
over $R$, the ring of differential polynomials in $y$ and $z$ is
denoted by $R\{y,z\}$ and defined as $R\{y\}\{z\} \simeq
R\{z\}\{y\}$.

A differential ring $(R,\partial)$, defines a ring of (linear)
differential operators $R[\partial]$, which is a non-commutative
ring, an Ore polynomial ring \cite{Good}, where the multiplication
of two elements is the composition of the two differential
operators, that is defined by the rule $\partial r = \partial(r) +
r\partial$ for all $r \in R$. A differential operator $A\in
R[\partial]$ can be written in a unique way as
\[A = a_n \partial^n + a_{n-1}\partial^{n-1}+\ldots +a_1\partial+a_0, \mbox{ with } a_i\in R,\,\,\,a_n\neq 0.\]
Then the \emph{order} of $A$ is $n$, denoted  $\ord(A)=n$,  with the
convention $\ord (0)=-\infty$. This order function shares properties
with the degree of a polynomial: the order of a product is the sum
of the orders and the order of a sum is bounded by the maximal order
of the summands.  The coefficient of the highest order term of~$A$
is the \emph{leading coefficient} of $A$, denoted $\lc(A)=a_n$. We
say that an operator~$A\in R[\partial]$ of order $n$ is in
\emph{normal form} if it is monic, i.e.~it has $\lc(A)=1$ and the
coefficient of the term of order $n-1$ is zero, namely,
\[A = \partial^n + a_{n-2}\partial^{n-2} + \ldots +a_1\partial+a_0.\]
In   \cite{Mu2, Zheglov} these operators are called {\it elliptic}.
Differential operators with coefficients in a differential field can
be always brought to normal form, by using a change of variable and
conjugation by a function, which are the only two  possible
automorphisms.

Let us assume that the ring of constants of $R$ is a field $\coC$ of
zero characteristic. Since $R[\partial]$ is a non-commutative ring,
we can define the corresponding Lie-bracket on $R[\partial]$ as a
$\coC$-bilinear map $[\cdot,\cdot]$ defined by
\[[A, B] = AB - BA.\]
It is clear that $A$ and $B$ commute if and only if $[A,B] = 0$. We
also say that $[A,B]$ is the \emph{commutator} of $A$ and $B$. We
define the \emph{centralizer} of $A$ {in $R[\partial]$} as
\[\centr(A) = \{B \in R[\partial]\ \mid [A,B] = 0\}.\]
Since $[\cdot,\cdot]$ is a $\coC$-bilinear map, then $\centr(A)$ is
a $\coC$-module.

\medskip

{ The next lemma states a basic property of ordinary differential
operators, that for operators in normal form will be important to
develop the main results of this paper, for this reason a proof is
included.}

\begin{lemma}\label{lem:commutator_order}
    Let $A,B$ be two differential operators in $R[\partial]$ of orders $n$ and $m$ respectively. Then $\ord([A,B]) \leq n+m-1$. If $A$ and $B$ are both in normal form, then $\ord([A,B]) \leq n + m -3$.
\end{lemma}
\begin{proof}
Since $\lc(AB) = \lc(A)\lc(B) = \lc(BA)$, the term of order $n+m$ of
$[A,B]$ is always zero, implying $\ord([A,B])\leq n+m-1$. Moreover,
if $A$ and $B$ are in normal form, we can write $A = \partial^n +
\tilde{A}$ and $B = \partial^m + \tilde{B}$ where $\tilde{A}$ and
$\tilde{B}$ have orders bounded by $n-2$ and $m-2$ respectively.
Then
\[[A,B] = [\partial^n, \partial^m] + [\tilde{A}, \partial^m] + [\partial^n, \tilde{B}] + [\tilde{A}, \tilde{B}].\]
The first bracket always vanishes and the orders of the other three
brackets are bounded by $n+m-3$.
\end{proof}

Let us denote by $R((\partial^{-1}))$ the ring of
pseudo-differential operators in $\partial$ with coefficients in the
differential ring $R$, defined as in~\cite{Good}  {( see also
\cite{Olver}, Chapter 5 and \cite{Dikii})}
\[
R((\partial^{-1}))=\left\{
{ \sum_{-\infty < i \leq n}} a_i\partial^i\mid a_i\in  R ,
n\in\bbZ\right\},
\]
where $\partial^{-1}$ is the inverse of $\partial$ in
$R((\partial^{-1}))$,$\partial^{-1}\partial=\partial\partial^{-1}=1$,
and it satisfies the following commutation rule with elements $r \in
R$
$$\partial^{-1} r = \sum_{i=1}^{\infty} (-1)^{i-1}  \partial^{i-1}(r) \;\partial^{-i} \ .$$

Observe that
\[\centr(L)\subset R[\partial]\subset R((\partial^{-1})).\]
Given $A\in R((\partial^{-1}))$, with $a_n\neq 0$ then $n$ is called
the \emph{order} of $A$, denoted by $\ord(A)$ and $a_n$ is called
the \emph{leading coefficient} of $A$. The \emph{positive part} of
$A$ is defined as the differential operator
\begin{equation}
   A_+= \sum_{i=0}^n a_i\partial^i\in R[\partial] \ \textrm{ if } n\geq 0 \ , \ \textrm{ and } A_+=0 \ \textrm{ otherwise  }.
\end{equation}

As it is ususal, denote by $A_- = A - A_+$. Given $A,B\in
R((\partial^{-1}))$ of orders $m$ and $n$ respectively,  their
commutator $[A,B]$ is also defined and using the decomposition $A =
A_+ + A_-$, we can see that
\[[A, B] = [A_+,B_+] + [A_-, B_+] + [A_+, B_-] + [A_-,B_-],\]
which easily proves that Lemma~\ref{lem:commutator_order} extends to
pseudo-differential operators, see for instance \cite{Good}.

\section{Almost Commuting Operators}\label{sec:wilson}

Let us consider a differential ring $(R,\partial)$ whose ring of
constants $\coC$ is a field of zero characteristic.

Given a monic $n$-th order differential operator $L$ in
$R[\partial]$, it is well known that it has a unique monic $n$-th
root, {  which is the first order pseudo-differential operator with
leading term $\partial$}, denoted by $L^{1/n}$,  see for instance {
~\cite[Proposition~1.3.8]{Dikii} and ~\cite[Proposition~2.7]{Good}}.
In addition, $L^{1/n}$ determines the centralizer in
$R((\partial^{-1}))$ of $L$, denoted by $\centr((L))$ and equal to
\begin{equation}\label{eq-SchurThm}
    \centr((L))=\left\{ \sum_{-\infty<j\leq n} c_j (L^{1/n})^j\mid c_j\in {\bf C}, j\in\bbZ\right\},
\end{equation}
 by \cite[Theorem~3.1]{Good}.
This implies that $\centr((L))$ is commutative. This result is a
generalization of a famous theorem by I. Schur in \cite{Sch},
originally stated for analytic coefficients, to differential
operators with coefficients in a differential ring $R$. This theorem
has a long history, see for instance~\cite[Sections~3 and~4]{Good},
and the references therein.

An immediate consequence of \eqref{eq-SchurThm}, given in the next
lemma, motivates the definition of almost commuting operators,
originally introduced by G. Wilson in \cite{SW, Wilson1985}. Let us
consider the set $\centr ((L))_+$ of all positive parts of the
pseudo differential operators in $\centr ((L))$, that is
\begin{equation}\label{eq-cen+}
    \centr ((L))_+:=\{ {B_+\mid B}\in \centr ((L))\}.
\end{equation}
Observe that $ \centr ((L))_+ $ is a $\coC$-linear space.
\begin{lemma}\label{lema-order-bracket}
    Let us consider a monic differential operator $L\in R[\partial]$ of order~$n$. For every  $A\in \centr ((L))_+$ it holds that $\ord([L,A])\leq n-2$.
\end{lemma}
\begin{proof}
    { Given $A\in \centr ((L))_+$ then $A=B_+$, with $B\in \centr ((L))$}. Let us consider the natural decomposition $B=B_+ + B_-$ and observe that $B_-$ is a pseudo-differential operator of order at most $-1$. Since $B\in\centr((L))$ then $0=[L,B]=[L,B_+]+[L,B_-]$ implying that the order of $[L,B_+]$ equals the order of $[L,B_-]$. This order is less than or equal to $n-2$.
\end{proof}

The next definition is naturally motivated by the previous lemma.

\begin{definition}[\cite{Wilson1985}]
    Let $L \in R[\partial]$ be a differential operator of order $n$. We say that $A \in R[\partial]$ \emph{almost commutes with $L$} if $\ord([L,A]) \leq n-2$. We denote by $W(L)$ the set of all differential operators that almost commute with $L$.
\end{definition}

\begin{lemma}\label{lem-WL}
    Let $(R,\partial)$ be a differential ring, whose ring of constants is a field~$\coC$ of zero characteristic. Given a differential operator $L\in R[\partial]$ the following statements hold.
    \begin{enumerate}
        \item $W(L)$ is a $\coC$-vector space.

        \item Given $A\in W(L)$ then $\lc (A)\in \coC$.
    \end{enumerate}

\end{lemma}
\begin{proof}
    For every $B_1,B_2 \in W(L)$ and $c\in \coC$ using the $\coC$-linearity of the Lie bracket:
    \[[L, cB_1 + B_2] = c[L, B_1] + [L,B_2].\]
    Thus, by the properties of the order, $cB_1 + B_2 \in W(L)$, proving the first statement.

Given $A\in W(L)$, let us prove that $\lc(A)$ is a constant. If
$\ord(A)=m$, let $a_m=\lc(A)$. It is easily proved that the leading
term of $[L,A]$ is $n a_m' \partial^{n+m-1}$. Since $\ord([L,A])\leq
n-2$ then $a_m'=0$, proving that $a_m\in \coC$.
\end{proof}

Lemma \ref{lema-order-bracket}  implies that $\centr
((L))_+\subseteq W(L)$. We will next prove that equality holds. {
Given the monic  $n$-th root} $L^{1/n}$ of $L$ we denote its $m$-th
power by $L^{m/n}$.

\begin{prop}\label{prop-Cbasis}
   Let us consider a monic differential operator $L\in R[\partial]$ of order $n$ and  { its monic $n$-th root} $L^{1/n}$. Then the set $\cB(L)=\{B_j=(L^{j/n})_+\mid  j\in \bbN\}$ is a basis of $\centr ((L))_+$ as a $\coC$-vector space. 
\end{prop}
\begin{proof}
    The existence of $B=L^{1/n}$ is guaranteed by \cite{Good}, Proposition 2.7. Note that $B$ has order one. Thus $B_0=1$ and  $\ord(B_j)=j$, for every $j\in\bbN\backslash\{0\}$. Thus $\cB(L)$ is a set of linearly independent differential operators and, by \eqref{eq-cen+}  and  \eqref{eq-SchurThm}, it is a generating set of $\centr ((L))_+$ as a $\coC$-vector space.
\end{proof}

\begin{thm}\label{thm-pcen}
     Let $L\in R[\partial]$ be monic. Then $\centr ((L))_+= W(L)$.
\end{thm}
\begin{proof}
Assuming that $L$ has order $n$, we consider the $\coC$-basis
$\cB(L)=\{B_j=(L^{j/n})_+\mid  j\in \bbN\}$ of $\centr ((L))_+$ as
in Proposition \ref{prop-Cbasis}. By Lemma \ref{lem-WL}, given $A\in
W(L)$ of order $m$ then there exists $c_m\in \coC$ such that $A-c_m
B_m$ belongs to $W(L)$ and has order less than or equal to $m-1$.
Now by induction on the order of $A$, we conclude  that $A\in \centr
((L))_+$.
\end{proof}

\color{black}

\subsection{Formal differential operators}\label{sec-formal-operators}

Let us consider a set $U=\{u_2,\ldots u_n\}$ of differential
variables with respect to a derivation $\partial$ and the ring of
differential polynomials $\cR=\coC\{U\}$,  with the extended
derivation
\begin{equation}\label{def-w-cR}
    \partial^q (u_\ell ) =u_\ell^{(q)} , \quad q=1, 2, \dots .
\end{equation}
Then $\cR$ is a differential ring whose ring of constants equals
$\coC$. We call $\cR[\partial ]$ {\it the ring of formal
differential operators in the differential variables $U$} and its
elements {\it formal differential operators}.

\medskip


\begin{definition}
({\bf Weight function}) We define a weight function $w$ on $\cR$
given by $w(u_{\ell}) = \ell$, and
\begin{equation}\label{eq-weight}
w(u_\ell^{(k)})=\ell+{ k} , \ k=1, 2, \dots .
\end{equation}
As in \cite[Proposition~4.2]{SW}, we call an operator $P\in
\cR[\partial]$ \emph{homogeneous of weight $r$ with respect to $w$}
if the coefficient of $\partial^i$ in $P$ has weight $r-i$. This is
equivalent to extending $w$ to $\cR[\partial]$ by setting
$w(\partial) = 1$. From the homogeneity of the commutation rule
$\partial p= p\partial +\partial (p)$, for every $p\in \cR$ it
follows that the product of two operators that are homogeneous of
weights $r$ and $s$ is homogeneous of weight $r + s$.
\end{definition}

\medskip

Let $L$ be the formal differential operator in normal form
\begin{equation}\label{def-L-formal-U}
 L = \partial^n + u_2 \partial^{n-2} + \ldots + u_{n-1}\partial + u_n \in \cR[\partial],
\end{equation}
which is homogeneous of weight $n$.

{
\begin{rem}\label{rem-Q}
    An $n$-th root of $L$ in $\cR ((\partial^{-1}))$ is an order one pseudo-differential operator $\xi \partial +q_0+ q_{-1} \partial^{-1}+\cdots$, where $\xi$ is an $n$-th root of unity. We work with the unique $Q$ starting with $\partial$ such that $Q^n=L$. The coefficient of the term of order $n-1$ of $Q^n$ is $n q_0$. Since $L$ is in normal form, then the term of order $n-1$ of $L$ is zero, implying that $q_0=0$. Thus $Q$ is also in normal form
\begin{equation}\label{eq-Q}
Q=\partial +q_{-1} \partial^{-1}+\cdots.
\end{equation}
\end{rem}
}

The next theorem is our reformulation
of~\cite[Proposition~2.4]{Wilson1985}, which establishes the
foundation of the contributions of this paper.

\begin{theorem}\label{thm-basis}
Let $L$ be as in \eqref{def-L-formal-U}, and let $Q$ be its unique
monic $n-$th root in $\cR ((\partial^{-1}))$. Then,
    the  family of formal differential operators in $\cR[\partial]$
    \begin{equation}\label{def-basis-Pm}
        \cH(L)=\{P_j=(Q^j)_+\mid j\in \bbN\}
    \end{equation}
    is the unique $\coC$-basis of $W(L)$, up to multiplication by constants, verifying the following \emph{homogeneity condition}: each $P_j$ is a formal differential operator homogeneous of weight $j$. In addition, each $P_j$ is monic and in normal form. 
\end{theorem}
\begin{proof}
We say that a pseudo-differential operator in $\cR((\partial^{-1}))$
is homogeneous of  weight $r$ if the coefficient of $\partial^i$ is
homogeneous of weight $r-i$, with $i\in \bbZ$. In
$\cR((\partial^{-1}))$, it also holds that the product of two
operators that are homogeneous of weights $r$ and $s$ is homogeneous
of weight $r + s$. It follows immediately that if a
pseudo-differential operator is homogeneous of weight $r$ then its
positive part is also homogeneous of weight $r$.

By \cite[Proposition~4.2]{SW}, the unique monic $n$-th root
$Q=L^{1/n}$ of $L$ is a pseudo-differential operator in
$\cR((\partial^{-1}))$ and it is homogeneous of weight $1$.
Therefore $Q^j$ is homogeneous of weight $j$ and we conclude that \
$P_j=(Q^j)_+$ \break is homogeneous of weight $j$.  Since $L$ is
monic and in normal form, then { by Remark \ref{rem-Q}} so is $Q$,
which implies that each $P_j$ is monic and in normal form.

To prove its uniqueness, let us consider any other $\coC$-basis
$\mathcal{T} = \{T_j\mid j\in\bbN\}$ of $W(L)$ satisfying the
homogeneity condition. It is necessary that $T_j=c_j P_j$, with
$c_j\in \coC$, since otherwise it would not be homogeneous anymore.
If we  require every element of $\cT$ to be monic, we  conclude that
$c_j = 1$ for all $j \in \bbN$, proving $\mathcal{T} = \cH(L)$.
\end{proof}

We call $\cH (L)$, given in Theorem \ref{thm-basis}, {\it the
homogeneous basis} of $W (L)$. Let us denote by $W_M(L)$ the set of
all operators that almost commute with $L$ and have order less than
or equal to $M$. Observe that $W_M(L)$ is a $\coC$-linear subspace
of $W(L)$ and
\begin{equation}
    \cH_M (L)=\{P_1,\ldots ,P_M\}
\end{equation}
is a $\coC$-basis of $W_M(L)$. Our goal is to compute the basis
$\cH_M(L)$.

\begin{ex}\label{ex-n3}
Let us illustrate the homogeneous almost commuting basis in the
case~$n=3$, that leads to the~\emph{Boussinesq} hierarchy, which
will be studied in Section \ref{sec-Third}. Consider a~\emph{formal}
third-order differential operator
\begin{equation}\label{eq-Bsq-L3}
 L_3 =\partial^3 +u_2 \partial +u_3.
\end{equation}
Its monic cubic root is the pseudo-differential operator in $\cR
((\partial^{-1}))$
\begin{equation}\label{eq-cubic-root}
Q_3=L_3^{1/3}=\partial+\frac13 u_2 \partial^{-1}+\frac13
\left(u_3-u_2'\right)\partial^{-2}+\cdots
\end{equation}
that can be computed recursively comparing both sides of the
formula~$L_3=Q_3\cdot Q_3\cdot Q_3$.

The positive part of the pseudo-differential operator $Q_3$ is
$P_1=\partial$. The remaining operators of the almost commuting
homogeneous basis $\cH(L_3)$ of $L_3$ are
\begin{equation}\label{eq-Ai-GD}
 P_j:= (Q_3^j)_+ \ ,   \textrm{ with } j\geq 2.
\end{equation}
 The differential operator $P_j$ defined in \eqref{eq-Ai-GD} is a monic operator of order $j$ whose coefficient in $\partial^{j-1}$ is $0$. For $j \equiv  0\,(\mod\ 3)$ observe that $P_j=L_3^j$.
\end{ex}

In Section~\ref{sec:almost}, we provide  a deterministic algorithm
to compute the sequence ~$\{P_j\}$ defined in \eqref{def-basis-Pm}
up to a fixed bound $M$, for any formal differential operator of
order $n$ as in~\eqref{def-L-formal-U}. Our philosophy is to avoid
computations in the ring of pseudo-differential operators $\cR
((\partial^{-1}))$, since the formal implementation of $\cR
((\partial^{-1}))$ in a computational algebra system is not
available so far.

\section{Basis of almost-commuting ODOs}\label{sec:almost}

Let us consider a differential field of constants $(\coC,\partial)$
and assume that $\coC$ has zero characteristic, thus $\coC$ contains
(a field isomorphic to) the field of rational numbers $\bbQ$. Let
$L$ be the formal differential operator considered
in~\eqref{def-L-formal-U}
\begin{equation*}
 L = \partial^n + u_2 \partial^{n-2} + \ldots + u_{n-1}\partial + u_n \in \cR[\partial],
\end{equation*}
which belongs to $\bbQ\{U\}[\partial ]\subset \cR[\partial ]$, with
$\cR = \coC \{ U \}$ and $U = \{u_2,\ldots, u_n\}$. Let $\cBa (L)$
be the set of almost-commuting operators with $L$, and  $\cBa_M (L)$
its $\coC$-linear subspace whose operators have order $\leq M$.
Observe that $\cBa_M (L)$ is also a set of formal differential
operators in $\cR[\partial ]$. We compute next a $\coC$-basis of
$\cBa_M (L)$ as a consequence of the following construction.
\medskip

\begin{definition}
({\bf Extended weight function}) Let us define a new set of
differential variables $Y =\{ y_2 , \dots , y_M \}$ over $\cR$. Let
$Y_m$ denote $\{y_2,\ldots,y_m\}$ for $m=2,\ldots,M$. Extending the
derivation in $\cR$ formally by
\[
\partial^k ( y_\ell ) = y_\ell^{(k)} , \quad k=1, 2, \dots ,
\]
a new differential ring $\cR\{ Y \}$ can be considered. Given the
weight function $w$ on $\cR$ defined in \eqref{eq-weight}, an
extended weight function $\Tilde{w}$ can be similarly defined over
$\cR\{ Y \}$ by setting:
\begin{equation}
    \Tilde{w} \left( \,y_\ell^{(k)} \right)= \ell +k.
\end{equation}
Abusing notation, we denote this extension by $w$ when there is no
possible confusion.
\end{definition}

\medskip

Let us consider a fixed number $2 \leq m \leq M$. We define the
formal differential operator
\begin{equation}\label{def-Tilde_Pm}
   \tilde{P}_{m} := \partial^m +y_2 \partial^{m-2} + \cdots + y_{m-1 }\partial +y_m\ \in \bbQ\{Y_m\}[\partial]
\end{equation}
which is homogeneous of weight $m$. Consequently, the Lie bracket of
$L$ and $\tilde{P}_{m} $ is homogeneous of weight $n+m$ in
$\bbQ\{Y_m\}[\partial]$. By Lemma \ref{lem:commutator_order}, the
Lie bracket~$[L,\tilde{P}_m]$   provides a formal differential
operator of order { less than or equal to} $n+m-3$
since both operators are in normal form. The next  equation will be
of importance in the remaining parts of this paper:
\begin{equation}\label{def-lie-bracket}
[L,\tilde{P}_m] = E_{m,0}(U, Y_m )  + E_{m,1}(U, Y_m ) \partial +
\ldots + E_{m,n+m-3}(U, Y_m ) \partial^{n+m-3},
\end{equation}
where $E_{m,n+m-i}(U, Y_m)$ is a homogeneous polynomial of weight
$i$. Our aim is to solve the system
\begin{equation}\label{eq-system}
   \bbE_{m,n}:=\{ E_{m,n+m-i}(U, Y_m)=0\mid i=3, \dots ,  {m+1}\}
\end{equation}
w.r.t.~the variables $Y_m$ to obtain an operator of weight $m$ in
$\cR[\partial]$ that almost commutes with $L$ in $\cR[\partial]$,
leading to the element of order $m$ of the homogeneous basis of
$W(L)$ {defined in Theorem \ref{thm-basis}}.

\medskip

For a fixed set of differential polynomials $ Z= \{ q_2 , \dots ,
q_m \} \subset \cR$, with weights $w(q_{\ell})=\ell$, we define a
{differential} ring homomorphism, called  \emph{evaluation  by $Z$}
\begin{equation}
    \epsilon_Z : \cR \{ Y_m \} \rightarrow \cR \ , \
{\epsilon_Z(y_\ell^{(k)})=q_{\ell}^{(k)}},
\end{equation}
{and denote $f(U,Z):=\epsilon_Z   ( f(U,  Y_m ) )$} . Abusing
notation, we define the extension of the previous morphism to the
corresponding rings of differential operators
\begin{equation}
    \epsilon_Z : \cR \{ Y_m \} [\partial ]\rightarrow \cR [\partial ] \ , \
     \epsilon_Z
    \left( \sum_{i=0}^d  f_i \partial^i  \right)=\sum_{i=0}^d \epsilon_Z (f_i )\partial^i   .
\end{equation}

{ The next theorem explains that solving the system $\bbE_{m,n}$ in
\eqref{eq-system} for the set variables $Y_m$ is equivalent to
finding the homogeneous basis $\cH (L)$ of $W(L)$ as in Theorem
\ref{thm-basis}. Moreover, by considering weights, the solution $Z=
\{ q_2 , \dots , q_m \}$ is a unique set of differential polynomials
in $\cR = \coC \{ U \}$.}

\begin{thm}  \label{lem-almost-system}
{ Consider a formal differential operator $L$ of order $n$ in normal
form as in~\eqref{def-L-formal-U}.} Let us fix a positive integer
$m\geq 2$, and the homogeneous ope\-rator $P_m$ in the homogeneous
basis $\cH(L)$ of $W(L)$ (see Theorem \ref{thm-basis}). Let us
consider $Z= \{ q_2 , \dots , q_m \} \subset \cR$ with weights
$w(q_{\ell})=\ell$.
 The following statements are equivalent:
    \begin{enumerate}
    \item \label{thm-cond-1}  $\epsilon_Z ( \tilde{P}_m )=P_m $.

    \item\label{thm-cond-2} $Z$ is a solution of system $\bbE_{m,n}$ { in \eqref{eq-system} w.r.t $Y_m$}, that is
    \[E_{m,n+m-j}(U,Z) =0,\ j=3, \dots ,   \ m+1.\]
    \end{enumerate}
\end{thm}
\begin{proof}
The result follows from the next observation
\begin{equation}\label{eq-Ebracket}
     [L,\epsilon_Z (\tilde{P}_m )]
     =\epsilon_Z ( [L,\tilde{P}_m]  ) =
        \sum_{j=3}^{n+m}   E_{m,n+m-j}(U, Z ) \partial^{n+m-j} .
    \end{equation}
If we assume \ref{thm-cond-1}, since $[L, P_m] $ is an operator of
order $n-2$, then  $Z$ is a solution of system \eqref{eq-system}.
Now assume \ref{thm-cond-2}. Then, by \eqref{eq-Ebracket},
$\epsilon_Z (\tilde{P}_m )$ almost commutes with~$L$ and it is an
homogeneous operator of weight $m$, by the weight condition
on~$q_\ell$. Therefore, by the uniqueness in Theorem
\ref{thm-basis}, the operator $\epsilon_Z (\tilde{P}_m )$ equals
$P_m$.
\end{proof}

The existence of a unique solution of system $\bbE_{m,n}$ with
respect to the variables $Y_m$, is guaranteed by Theorem
\ref{lem-almost-system} and Theorem \ref{thm-basis}.

\begin{corollary}\label{cor-uniquesol}
For each $n,m\in \bbN$, $n,m\geq 2$, the system $\bbE_{m,n}$ has a
unique solution $Z=\{ q_2 , \dots , q_m \} \in \bbQ\{U\}$ in the
variables $Y_m$, {with weights $w(q_{\ell})=\ell$}. We call $Z$ the
\emph{weighted solution of $\bbE_{m,n}$}
\end{corollary}

\begin{algorithm}[!ht]
\caption{Algorithm to compute $P_m$}\label{alg:compute_basis}
\Input{Operator $L$ as in~\eqref{def-L-formal-U} and an integer $m \geq 2$}
\Output{Operator $P_m$ in $\cH(L)$ given in Theorem~\ref{thm-basis}}
Define $\tilde{P}_m$\; \texttt{LB} $\gets [L,\tilde{P}_m]$\;
$\bbE_{m,n}$ $\gets \left\{\coeff(\texttt{LB}, \partial^i)\ :\ i \in
\{n-1,\ldots, n+m-3\}\right\}$\; Compute the unique  weighted
solution $Z$ of system $\bbE_{m,n}$ in \eqref{eq-system}\; \Return
$\epsilon_Z (\tilde{P}_m )$\;
\end{algorithm}

The previous results induce  Algorithm ~\ref{alg:compute_basis} to
compute the almost commuting operator $P_m$ of the basis $\cH(L)$.
Iterating Algorithm~\ref{alg:compute_basis}  for $m=2,\ldots,M$ and
setting $P_1 = \partial$, we can compute the homogeneous basis of
$\cBa_M (L)$. We will discuss step \texttt{4} of this algorithm in
the remaining parts of this section.

\subsection{Triangular differential system}

The goal of this section is to prove that the system
\[\bbE_{m,n}=\{E_{m,m+n-i}(U,Y_m)\mid i=3,\ldots ,m+1\}\] is a
triangular system in the set of differential variables
$Y_m=\{y_2,\ldots ,y_m\}$. In addition, each differential polynomial
$E_{m,m+n-i}(U,Y_m)$ will be proved to be homogeneous of  weight
$i$.

Let us observe that we can write the operators $\tilde{P}_m$ of
\eqref{def-Tilde_Pm} recursively as
\begin{equation}\label{def-Pm-recursion}
     \tilde{P}_{\ell +1}  =  \tilde{P}_{\ell }  \partial + y_{\ell +1},\,\,\, \ell \geq 1,\,\, \textrm{ defining }  \tilde{P}_{1} := \partial.
\end{equation}
This induces the following recursive formula to compute the
Lie-brackets $[L,\tilde{P}_m]$.

\begin{lemma}\label{lem:recursive_bracket}
    Let $L$ be as in~\eqref{def-L-formal-U}. Then, for $\ell \geq 1$ \begin{equation}\label{equ:recursive_bracket}
       [L, \tilde{P}_{\ell+1}] =  [L, \tilde{P}_\ell ]\partial + \tilde{P}_\ell [L,\partial] + [L, y_{\ell+1}].
    \end{equation}
\end{lemma}
\begin{proof}
    Using the $\bbQ$-linearity of the commutator and formula~\eqref{def-Pm-recursion}, we can compute:
\begin{align*}
    [L, \tilde{P}_{\ell+1}] & = [L, \tilde{P}_{\ell}\partial] + [L, y_{\ell+1}] \\
                         & = L\tilde{P}_{\ell}\partial - \tilde{P}_{\ell}\partial L +  [L, y_{\ell+1}]\\
                         & = L\tilde{P}_{\ell}\partial - \tilde{P}_\ell L \partial + \tilde{P}_\ell L \partial - \tilde{P}_{\ell}\partial L + [L, y_{\ell+1}]\\
                         & = [L, \tilde{P}_\ell ]\partial + \tilde{P}_\ell [L,\partial] + [L, y_{\ell+1}]
\end{align*}
\end{proof}

We can see that the first summand adds the coefficients of
~$\partial^i$ in $[L, \tilde{P}_\ell]$ as coefficients of
$\partial^{i+1}$ in $[L, \tilde{P}_{\ell+1}]$. The second summand
can be easily computed as a product of differential operators (see
Lemma~\ref{lem:bracket_L_partial}). Finally, the third summand has
also a very generic formula (see Lemma~\ref{lemma-PlusOne}) and is
the \emph{only} summand that involves the differential variable
$y_{\ell+1}$.

\begin{lemma}\label{lem:bracket_L_partial}
    Let $L$ as in~\eqref{def-L-formal-U} and $\tilde{P}_\ell$ as in~\eqref{def-Pm-recursion}. Then:
    \begin{enumerate}
        \item \label{lem-2-1}$[L,\partial] =-\left( u_n' + u_{n-1}'\partial + \ldots + u_2'\partial^{n-2}\right)$. Therefore, it is a homogeneous operator of weight $n+1$.
        \item \label{lem-2-2}For $\ell \geq 2$, the differential operator $Q_\ell:=\tilde{P}_\ell [L,\partial]$ has order $n+\ell-2$, and it is a homogeneous operator of weight $\ell+n+1$. In particular, the coefficients of $Q_\ell$ only contain the differential variables $y_2,\ldots,y_\ell$ with no derivatives. Moreover, the coefficient of $y_{\ell-i}\partial^k$ in $Q_\ell$ is $0$, for $k> n+i-2$.
    \end{enumerate}
\end{lemma}

\begin{proof}
   The formula in \ref{lem-2-1} follows from the commutation rule $ \partial u_k = u_k \partial +u_k '$. Next, for proving~\ref{lem-2-2}, we first observe that $\ord(Q_\ell) = \ell + n -2$, and, since $\tilde{P}_\ell$ is homogeneous of weight $\ell$ and $[L,\partial]$ is homogeneous of weight $n+1$, then $Q_\ell$ is also homogeneous of weight $\ell + n + 1$. Moreover, we can compute
   \begin{align*}
     Q_\ell & =
     \left(
     \partial^\ell + \sum_{i=0}^{\ell-2} y_{\ell-i} \partial^i
     \right) \cdot
     \left(
     -\sum_{j=0}^{n-2} u_{n-j}' \partial^j
     \right) = \\
   & =-\sum_{j=0}^{n-2} \partial^\ell u_{n-j}' \partial^j  -
   \sum_{i=0}^{\ell-2} y_{\ell-i}
   \left(
   \sum_{j=0}^{n-2}
   \sum_{s=0}^i \binom{i}{s}
   u_{n-j}^{(i-s+1)} \partial^{j+s}
   \right)= \\
   & = - u_2' \partial^{n+\ell-2}-
 \left( \ell u_2'' +u_3'\right)\partial^{n+\ell-3 }- \sum_{i=0}^{\ell-2} y_{\ell-i}
 H_{n+i-2}+    \left(\textrm{l.o.t.~in } \coC \{ U \} [\partial ]\right),
 \end{align*}
with $H_{n+i-2} $ a ordinary differential operator in $\coC \{ U \}
[\partial ]$ of order at most $n+i-2 $. Then, the result follows.
\end{proof}

\begin{lemma}\label{lemma-PlusOne}
Let $\ell$ be a positive integer.  Then $[L, y_{\ell +1}]$ is a
homogeneous diffe\-rential operator of order $n-1$, weight $n+\ell
+1 $,  and leading coefficient $n y_{\ell+1}'$.
\end{lemma}

\begin{proof}
Computing the commutator  $[L, y_{\ell +1}]= L y_{\ell +1} - y_{\ell
+1} L$, we obtain
\begin{align*}
 [L, y_{\ell +1}] &=  \displaystyle\left(\partial^n + \sum_{i=0}^{n-2} u_{n-i}\partial^i\right) y_{\ell + 1} - \left(y_{\ell + 1}\partial^n + \sum_{i=0}^{n-2} y_{\ell + 1} u_{n-i}\partial^i\right) \\
                  &=  \sum_{k=0}^{n-1} \binom{n}{k} y_{\ell + 1}^{(n-k)}\partial^k + \sum_{i=0}^{n-2}\sum_{k=0}^{i} \binom{i}{k} u_{n-i}y_{\ell + 1}^{(i-k)}\partial^k
                  - \sum_{k=0}^{n-2} y_{\ell + 1} u_{n-k}\partial^k\\
                  &=  \sum_{k=0}^{n-1} \binom{n}{k} y_{\ell + 1}^{(n-k)}\partial^k
                  + \sum_{k=0}^{n-2}\sum_{i=k}^{n-2} \binom{i}{k} u_{n-i}y_{\ell + 1}^{(i-k)}\partial^k
                  - \sum_{k=0}^{n-2} y_{\ell + 1} u_{n-k}\partial^k\\
                  &=  n y_{\ell + 1}'\partial^{n-1} + \displaystyle\sum_{k=0}^{n-2} \left[\binom{n}{k} y_{\ell + 1}^{(n-k)} + \sum_{i=k+1}^{n-2}\binom{i}{k} u_{n-i}y_{\ell + 1}^{(i-k)}\right]\partial^k\\
                  &=  n y_{\ell + 1}'\partial^{n-1} + \textrm{l.o.t.}
\end{align*}
  The homogeneous weight condition follows from the fact that $L$ is homogeneous of weight $n$ and $y_{\ell + 1}$ is homogeneous of weight $\ell + 1$.
\end{proof}

We prove next the triangularity of the system $\bbE_{m,n}$
defined in \eqref{eq-system}. We provide algorithms to solve the
system $\bbE_{m,n}$ in Section \ref{sec-integration}.

\begin{theorem}[Triangular system]\label{thm-triangularE}
Given $L$ as in~\eqref{def-L-formal-U} of order $n$ and $m\geq 2$,
let us consider the system $\bbE_{m,n}$. Then, for every $i =
3,\ldots,m+1$,  the differential polynomial
\begin{equation}\label{eq-diff-syst1}
       E_{m,m+n-i} (U,  Y) = n y_{i-1}' + e_{m,i-2},
\end{equation}
where  $e_{m,i-2} \in \cR\{Y_{i-2}\}$  is homogeneous of weight $i$,
for $i\geq 3$ (setting $Y_1 = \emptyset$). More precisely,
$E_{m,m+n-3}= ny'_2 -m u'_2$.
\end{theorem}

\begin{proof}

We proceed by induction on $m$. First, for $m=2$ we use
Lemmas~\ref{lem:recursive_bracket} and~\ref{lem:bracket_L_partial}
to we obtain
\[
[L,\tilde{P}_{2}] = -2\sum_{p=1}^{n-1} u'_{p+1} \partial^{n-p} -
\sum_{p=2}^{n} u''_p \partial^{n-p} +[L,y_2 ] .
\]
Then, $\bbE_{2,n} = \{E_{2,n-1}\}$, which is the coefficient of
$\partial^{n-1}$. Therefore,
\[
 \bbE_{2,n} = \{ny'_2 { -2 u'_2 }\}.
\]
 Now, assume the theorem holds for $m \geq 2$ and let us prove the case $m+1$. Using the recursive structure of $[L, \tilde{P}_{m+1}]$ from Lemma~\ref{lem:recursive_bracket},  formula \eqref{def-Tilde_Pm} and Lemma \ref{lem:bracket_L_partial}, the following recursion holds
 \begin{equation*}
 \label{eq-recursion-corchete-1}
   [L, \tilde{P}_{m+1}]= \! \sum_{j=0}^{n+m-3}\!\!  E_{m,j} \partial^{j+1}-
   \Bigg(
   \partial^m +\sum_{j=0}^{m-2} y_{m-j} \partial^j
   \Bigg)
   \Bigg( \sum_{p=2}^n u'_p \partial^{n-p}
   \Bigg) +
   [L, y_{m+1}].
 \end{equation*}
We use the following notations to simplify the previous identity:
\[\left\{\begin{aligned}
        A_j &:= \partial^j \displaystyle\sum_{p=2}^n u'_p \partial^{n-p} \in \cR [\partial ] \ , \textrm{ of order }  n+j-2 ,\\
        \alpha_{j,\ell} &:= \coeff(A_j, \partial^\ell) \in \cR,\vspace{1em}\\
        R_{m+1} &:= [L, y_{m+1}]-n y'_{m+1} \partial^{n-1} \in \cR\{y_{m+1}\}[\partial] \text{ of order } n-2.
\end{aligned}\right.\]
Hence, the previous identity can be rewritten as
\[
    [L, \tilde{P}_{m+1}]= \sum_{j=0}^{n+m-3} E_{m,j} \partial^{j+1}-
A_m - \sum_{j=0}^{m-2} y_{m-j} A_{j} + ny'_{m+1} \partial^{n-1} +
R_{m+1} .
\]

In order to consider $\bbE_{m+1,n}$, we need to look the
coefficients of $\partial^{n+m+1-\ell}$, for $\ell = 3, \ldots,
m+2$. Let us begin with the two extreme cases. For $\ell = m+2$:
\[E_{m+1, n-1} = E_{m, n-2} - \alpha_{m,n-1} - \sum_{j=0}^{m-2} y_{m-j} \alpha_{j, n-1} + ny'_{m+1}.\]
Then  $e_{m+1, m} := E_{m+1, n-1} - ny'_{m+1} $ is in $
\cR\{Y_m\}$. On the other hand, for $\ell = 3$ we have:
\[E_{m+1, n+m-2} = E_{m, n+m-3} - \alpha_{m,n+m-2} - \sum_{j=0}^{m-2} y_{m-j} \alpha_{j, n+m-2}.\]
Since $A_j$ has order $n+j-2$, it is clear that $\alpha_{j, \ell} =
0$ for $\ell > n + j - 2$, so, in particular, $\alpha_{j, n+m-2} =
0$ for $j < m$. Moreover, we can observe that~$\alpha_{j,n+j-2} =
u_2'$. Hence, using the recursion for $E_{m, n+m-3}$ we get:
\[E_{m+1, n+m-2} = ny_2' - mu_2' - u_2' = ny_2' - (m+1)u_2'.\]

Finally, for all the intermediate equations, $4\leq \ell \leq m+1$,
we observe that
\[
E_{m+1, n+m+1-\ell} = E_{m, n+m-\ell} - \alpha_{m, n+m+1-\ell} -
\sum_{j=m+3-\ell}^{m-2} y_{m-j} \alpha_{j, n+m+1-\ell}.
\]
But, by the induction hypothesis, the first summand equals $E_{m,
n+m-\ell} = ny_{\ell-1}' + e_{m, \ell-2}$, where $e_{m, \ell-2}$ is
in $  \cR\{Y_{i-2}\}$. The second summand is in $\cR$, and we can
observe that the third summand involve the variables $y_k$ for $k =
2,\ldots,\ell-3$. Observe that in the case of $\ell = 4$ this sum is
empty.

Therefore, $E_{m+1, n+m+1-\ell} = ny_{\ell-1}' + e_{m+1, \ell-2}$
with $e_{m+1, \ell-2} \in \cR\{Y_{\ell-2}\}$, and this concludes the
proof.
\end{proof}

\subsection{Solving by integration in closed form }\label{sec-integration}\label{sec-closedForm}

Let $L$ be the formal differential operator of order $n\geq 2$ in
$\bbQ\{U\}[\partial]$ and in normal form as in
\eqref{def-L-formal-U}. At this point, for a fixed $m\geq 2$, we
would like to solve the triangular differential system $\bbE_{m,n}$
with respect to the set of variables $Y_m=\{ y_2 , \dots , y_m\}$.
Recall that $Y_i = \{ y_2 , \dots , y_{i} \}$, $2\leq j\leq m+1$ and
define $Y_1:=\emptyset$. By Theorem \ref{thm-triangularE} the
equations of $\bbE_{m,n}$ have the form
\begin{equation}\label{eq-diff-syst}
       E_{m,m+n-i} (U,  Y) = n y_{i-1}' + e_{m,i-2} =0,\ i = 3,\ldots,m+1,
\end{equation}
with  $e_{m,i-2} \in \bbQ\{U\}\{Y_{i-2}\}$ homogeneous of weight
$i$. Furthermore, at each step, the new variable $y_{i-1}$, that we
solve for, appears differentiated only once.

The existence of a unique weighted solution $Z=\{q_2,\ldots ,q_m\}$,
with weights $w(q_i)=i$, of $\bbE_{m,n}$ with respect to $Y_m$ is
guaranteed by Corollary \ref{cor-uniquesol}. It implies the
following result:

\begin{lemma}\label{lem-totalder}
Let $Z=\{q_2,\ldots ,q_m\}$  be the unique weighted solution of
system $\bbE_{m,n}$ with respect to the variables $Y$.
    Then the polynomials
    $$e_{m,i}(q_2,\ldots ,q_{i-1}):=\epsilon_Z(e_i(Y_{i})),\,\,\, i=3,\ldots ,m$$ are total derivatives of polynomials in $\bbQ\{U\}$. Moreover, there exists a unique $q_i\in \bbQ\{U\}$, of weight $i$, such that $\partial(q_i)=e_{m,i}(q_2,\ldots ,q_{i-1})$.
\end{lemma}

More precisely, we have recursive formulas
\begin{equation}
    q_2:=\frac{m }{n}u_2 ,\ q_i:=\frac{1}{n}\partial^{-1} e_{m,i}(q_2,\ldots ,q_{i-1}),\ i=3,\ldots ,m.
\end{equation}
Observe that $\partial^{-1} e_{m,i}(q_2,\ldots ,q_{i-1})$ is well
defined by Lemma \ref{lem-totalder}.

\medskip

\begin{algorithm}[!ht]
\caption{\texttt{Integrate}}\label{alg:solving_integration}
    \Input{System $\bbE_{m,n}$ as in~\eqref{eq-system}}
    \Output{Unique weighted solution $Z=\{q_2,\ldots,q_m\}$ of $\bbE_{m,n}$ w.r.t.~$Y$}
    $Z \gets \{y_2 = \frac{m }{n}u_2\}$\;
    \For{$i \gets 3,\ldots,m$}{
        $e_{m,i} \gets \epsilon_Z(E_{m,n+m-i-1}) - ny_i'$\Comment*[r]{As in Equation~{\ref{eq-diff-syst}}}
        $q_i \gets \frac{1}{n}\texttt{INTEG}(e_{m,i})$\;
        $Z \gets Z \cup \{y_i = q_i\}$\;
    }
    \Return $Z$\;
\end{algorithm}

In step 4 of Algorithm \ref{alg:solving_integration} the instruction
 $\texttt{INTEG}(e_{m,i})$ indicates the computation of $\partial^{-1} e_{m,i}(q_2,\ldots ,q_{i-1})$. For this purpose we use the canonical decomposition of differential polynomials in one differential variable introduced in ~\cite{Bilge1992}. This method was  extended in the context of differential algebra in the article \cite{Boulier2016}, for partial derivations and general rankings.

Namely, given a differential polynomial $F(z)\in \bbQ\{z\}$, an
algorithm was  provided in ~\cite{Bilge1992} to write this
polynomial in a unique form~$F(z) = \partial(A(z)) + B(z)$, where
$B(z)$ is not in the image of~$\partial$ \cite{Bilge1992,
Boulier2016}, they contain only non-linear monomials in their
highest derivative. This algorithm works for multivariate
differential polynomials in $\bbQ\{U\}$ by iterating the algorithm
over each of the differential variables $u_2>u_3>\cdots >u_n$ and
collecting the final result.

The methods in~\cite{Bilge1992, Boulier2016} provide a criteria to
detect total derivatives (i.e., $F(z)$ is a total derivative if and
only if $B(z) = 0$). We have implemented this method to compute
antiderivatives of differential polynomials in our SageMath package,
\texttt{dalgebra} (see Section~\ref{ExperimentalResults}).

\section{Computing integrable hierarchies}\label{sec:hierarchies}

It is well known that the Korteweg-de Vries (KdV) equation $ u_t
=u_{xxx}-6u u_x$ can be represented as the integrability condition
of the linear differential system
\[
L \psi = \lambda \psi \ , \psi_t = A_3 \psi ,
\]
for a Schrödinger operator $L =-\partial^2 +u$ and a third order
operator $A_3 = \partial^3 +\frac{3}{2}u\partial +\frac{3}{4}u_x$,
where $\partial=\frac{d}{dx}$. In fact, the previous system can be
rewritten as a Lax equation
\[
L_t = [A_3 , L] \ ,
\]
where $L_t$ denotes the differential operator obtained from $L$ by
computing the partial derivative of its coefficients with respect to
$t$. For $m \geq 1$ an arbitrarily given natural number, the
equations of the KdV hierarchy appear by considering the Lax
equations
\[
L_t = [A_m , L] \ , \ \textrm{ with } t=t_m \ ,
\]
where $A_m$ is an almost commuting differential operator of order
$m$. Observe that  $[A_m, L]$ is a multiplication operator.

\medskip

Following \cite{DS, Dikii, KN2}, the Gelfand-Dickey hie\-rar\-chies,
or \ generalized KdV hierarchies in \cite{SW}, appear allowing the
operator $L$ to be a formal differential operator of arbitrary order
$n$, in normal form, as in Section \ref{sec:almost}
\begin{equation*}
         L = \partial^n + u_{2}\partial^{n-2} + \ldots + u_{n-1}\partial + u_{n} .
    \end{equation*}
Let us consider a { field} of constants $(\coC,\partial)$ and assume
that $\coC$ has zero characteristic, thus $\coC$ contains $\bbQ$.
 The Lax equations in this case are
 \begin{equation}\label{eq-Lax}
     L_t = [A_m , L] \ , \ \textrm{ with } t=t_m \ ,
 \end{equation}
 where
\begin{equation*}
    A_m= \partial^m + a_2 \partial^{m-2} + \ldots +a_{m-1}\partial + a_m ,
\end{equation*}
is an almost commuting operator with $L$ of order $m$, that is $A_m\in W(L)$.  The equation \eqref{eq-Lax} can be written in terms of the elements of the homogeneous 
$\coC$-basis $\cH(L)$ in \eqref{def-basis-Pm} of Theorem
\ref{thm-basis}. In fact,
\begin{equation}\label{eq-Lax-A}
    A_m =P_m +c_{m,m-1}P_{m-1}+\dots + c_{m,0}P_{0} ,\,\,\, c_{m,j} \in \coC.
\end{equation}
More precisely, let us consider a fixed $m \in \bbN$ and the
operator $P_m$ in the almost commuting basis obtained with Algorithm
~\ref{alg:compute_basis} and $Z = \{q_2,\ldots,q_n\}$ be the
elements that allow $\epsilon_Z(\tilde{P}_m) = P_m$. Since $P_m$
almost commutes with $L$, then $[L,P_m]$ is a differential operator
with order $n-2$. As described in Theorem~\ref{lem-almost-system},
we obtain:
\[[L,P_m] = [L,\epsilon_Z(\tilde{P}_m)] = \sum_{i=2}^{n} E_{m,n-i}(U,Z)\partial^{n-i}.\]
Let us denote the specialization of $-E_{m,n-i}(U,Y)$ by
$\epsilon_Z$ as
\[H_{m,n-i}(U) := -E_{m,n-i}(U,Z),\,\,\, m\geq 1, i=2,\ldots ,n.\]
Consequently, the Lax equations \eqref{eq-Lax}  provide a system of
partial differential equations defined by non linear differential
polynomials in the set of differential variables $U$, obtained from
the coefficient of $\partial^{n-i}$ in $[A_m,L]$
\begin{equation}\label{eq-diff-system-LA}
    u_{i, t} = H_{m,n-i}(U) +\sum_{j=1}^{m-1} c_{m,j}  H_{j,n-i}(U) \ , \textrm{for } \  i = 2, \dots , n,\,\,\, m\geq 2.
\end{equation}
For a fixed $n$, the family \eqref{eq-diff-system-LA} is called {\it
the Gelfand-Dickey (GD) hierarchy of $L_n$}, see \cite{Dikii,
DrinfeldSokolov1985, Wilson1985} and references therein. Some of
these hierarchies have been given specific names, for instance, for
$n=2$, the Gelfand-Dickey hierarchy of $L_2=\partial^2+u_2$ is the
Korteweg-de Vries (KdV) hierarchy
\begin{equation}\label{eq-KdV}
    u_{2, t} = H_{m,0}(U) +\sum_{j=1}^{m-1} c_{m,j}  H_{j,0}(U)\ ,\,\,\, m\geq 2.
\end{equation}
For $n=3$, the Boussinesq hierarchy is obtained, see Section
\ref{sec-Third}.

\medskip

Other methods have been developed to compute the GD hierarchies,
although normally they are particular methods for a fixed level, a
fixed value of~$n$. It is well known that the KdV hierarchy can be
computed using recursion operators. We can rewrite \eqref{eq-KdV} as
\begin{equation}\label{eq-KdVb}
    u_{2, t} = \kdv_m (u_2) +\sum_{j=1}^{m-1} c_{m,j}  \kdv_j (u_2)\ ,\,\,\, m\geq 2.
\end{equation}
where $\{\kdv_j(u_2)\}_{j\geq 0}$ is computed using the recursion
operator, see \cite{MRZ1} based on \cite{GH},
\[\cR=-\frac{1}{4}\partial^2+u_2+\frac{1}{2}u_2'\partial^{-1},\]
defining the recursion
\begin{equation}\label{eq-kdvc}
\kdv_0:=u',\,\,\, \kdv_n:=\cR(\kdv_{n-1}),\mbox{ for }n\geq 1.
\end{equation}

We obtain from \eqref{eq-diff-system-LA} {\it the stationary
Gelfand-Dickey hierarchy of $L_n$} consi\-dering the differential
variables in $U$  as constants with respect to $t$. These nonlinear
differential polynomials can be seen as conditions over the
differential variables $U$ for an almost commuting operator $A_m$ to
commute with $L_n$. See \cite{DrinfeldSokolov1985, Wilson1985} and
the references therein.

\color{black}

\subsection{Third order operators and their GD hierarchy}\label{sec-Third}

Let us consider the third order operator in $\cR[\partial]$ with
$\cR=\coC\{u_2,u_3\}$ given by
\[L_3 = \partial^3 + u_2\partial + u_3.\]
The Gelfand-Dickey hierarchy of $L_3$ consists of systems of two
nonlinear equations. For each $m\geq 2$
\begin{align}\label{eq-Bq_system}
    u_{2, t} = H_{m,1}(u_2,u_3) +\sum_{j=1}^{m-1} c_{m,j}  H_{j,1}(u_2,u_3),\\
    u_{3, t} = H_{m,0}(u_2,u_3) +\sum_{j=1}^{m-1} c_{m,j}  H_{j,0}(u_2,u_3).
\end{align}

We show the computation of the homogeneous basis $\cH_5(L_3)$ of the
space of almost commuting operators $W_5(L_3)$ of order less than or
equal to $5$,
\[\cH_5(L_3)=\{P_1, P_2, P_3,P_4,P_5\}.\]
For each $m\geq 1$, it holds that
\[ [P_m , L_3 ] = H_{m,1}(U)\partial + H_{m,0}(U),\]
where $H_{m,0}(U)$ and $H_{m,1}(U)$ are differential polynomials in
$\cR$ that determine all the equations of the Gelfand-Dickey
hierarchy for $L_3$. We will compute next the equations
\eqref{eq-Bq_system} for $2\leq m\leq 5$. Recall that $P_1=\partial$
by Example \eqref{ex-n3}, thus $[L_3, P_1] = - u_2'\partial - u_3'$.
Hence we have
\[H_{1,0}(U) = u_{3}',\,\,\, H_{1,1}(U) = u_{2}'.\]

\medskip

For $m=2$, the Lie bracket of $L_3$ with the polynomial $\tilde{P}_2
= \partial^2 + y_2$ equals
\[
    [L_3, \tilde{P}_2]=
                        E_{2,2} \partial^2 + E_{2,1} \partial  + E_{2,0},
\]
where
\begin{align*}
    E_{2,2}  & =   \mathbf{3y_2' - 2u_2'}, \\
    E_{2,1}  & =   3y_2'' -u_2'' -2u_3', \\
    E_{2,0}  & =   y_2''' + u_2y_2 - u_3''.
\end{align*}
In order to obtain the almost commuting operator $P_2$, we force the
coefficient of $\partial^2$ to be zero.  We compute $q_2$ such that
$E_{2,2}(u_2,q_2) = 0$ and obtain
\[P_2=\partial^2 + q_2=\partial^2 + \frac{2}{3}u_2.\]
The system of the GD hierarchy at level $m=2$ is,
\[u_{2,t}=H_{2,1}+c_{2,1} H_{1,1} ,\,\,\, u_{3,t}=H_{2,0}+c_{2,1} H_{1,0}\]
where
\[H_{2,0}(u_2,u_3) = -u_2'' + 2u_3',\qquad H_{2,1}(u_2,u_3) = -\frac{2}{3}u_2''' - \frac{2}{3}u_2'u_2 + u_3''.\]

\begin{rem}
The system of the GD hierarchy for $m=2$ with $c_{1,2}=0$ is
equivalent to the system of partial differential equations
\begin{equation}\label{eq:boussGD}
u_{3,t}=-u_{2}''+2u_{3}',\quad u_{2,t}=u_{3}''-\frac23
u_{2}'''-\frac23 u_{2}u_{2}'   \ , \ \textrm{with } t=t_2 .
\end{equation}
Hence, eliminating~$u_2$, implies the following form of the
classical Boussinesq equation
\begin{equation}\label{eq:bouss}
(u_3)_{tt}=-\frac13 u_{3}^{(4)}-\frac43 (u_{3}u_{3}')'.
\end{equation}

Recursive formulas were given in ~\cite{DGU} to provide an
elementary algebraic a\-ppro\-ach {for the entire hierarchy of}
Boussinesq systems.
The recursion operator for~\eqref{eq:bouss} can be found using the
technique given in~\cite{GKS} or through liftings of known recursion
operators for different forms of the Boussinesq equation. It is not
the goal of this paper to explain how all these forms of the
Boussinesq systems can be identified.

\end{rem}

\medskip

For $m=3$, we have $P_3=L_3$ and $H_{3,0}(U)=H_{3,1}(U)=0$. For $m =
4$, the Lie bracket of $L_3$ with the polynomial $\tilde{P}_4 =
\partial^4 + y_2 \partial^2+ y_3\partial+y_4$ equals
\[
[L_3, \tilde{P}_4]=E_{4,4} \partial^4+E_{4,3} \partial^3 + E_{4,2}
\partial^2  + E_{4,1}\partial+ E_{4,0},
\]
where
\begin{align*}
 E_{4,4} & =  \mathbf{3y_2' -4u_2'},\\
 E_{4,3} & =   \mathbf{3y_3' + 3y_2'' - 6u_2'' - 4u_3'},\\[0.25em]
 E_{4,2} & =  \mathbf{3y_4' + 3y_3'' + y_2''' + u_2y_2' - 2u_2'y_2 -4u_2''' -6u_3''} , \\[0.25em]
E_{4,1} & = 3y_4''+y_3'''+ u_2y_3'-u_2''y_2 - u_2'y_3  -2u_3'y_2 - u_2^{(4)} -4u_3''',\\[0.25em]
E_{4,0} & =  y_4''' + u_2y_4' - u_3''y_2 - u_3'y_3 - u_3^{(4)}.
\end{align*}
Solving the triangular system
$\bbE_{4,3}=\{E_{4,4}=0,E_{4,3}=0,E_{4,2}=0\}$ with respect to $\{y_2,y_3,y_4\}$ we obtain $Z=\{q_2,q_3,q_4\}$ 
and then
\[P_4=\epsilon_Z (\tilde{P}_4)=\partial^4 + \frac{4}{3}u_2\partial^2 + \left(\frac{2}{3}u_2' + \frac{4}{3}u_3\right)\partial + \left(\frac{2}{9}u_2'' + \frac{2}{3}u_3' + \frac{2}{9}u_2^2\right).
\]
The system of the GD hierarchy at level $m=4$ is,
\[u_{2,t}=H_{4,1}+c_{4,2} H_{2,1}+c_{4,1} H_{1,1} ,\,\,\, u_{3,t}=H_{4,0}+c_{4,2} H_{2,0}+c_{4,1} H_{1,0}\]
where
\begin{align*}
    -H_{4,0}(u_2,u_3) &= \frac{4}{9}u_2'u_2^2 + \frac{2}{3}u_2'''u_2 + \frac{4}{3}u_2''u_2' - \frac{2}{3}u_2'u_3' - \frac{2}{3}u_2u_3'' - \frac{4}{3}u_3'u_3 \\
    &\quad{}+ \frac{2}{9}u_2^{(5)}- \frac{1}{3}u_3^{(4)},\\
   -H_{4,1}(u_2,u_3) &= \frac{2}{3}u_2''u_2 + \frac{2}{3}u_2'^2 - \frac{4}{3}u_2'u_3 - \frac{4}{3}u_2u_3' + \frac{1}{3}u_2^{(4)} - \frac{2}{3}u_3'''.
\end{align*}

For $m =5$, the Lie bracket of $L_3$ with the polynomial
$\tilde{P}_5 = \partial^5 + y_2 \partial^3+ y_3\partial^2+
y_4\partial +y_5$ equals
\[
[L_3, \tilde{P}_5]=E_{5,5} \partial^5+E_{5,4} \partial^4+E_{5,3}
\partial^3+E_{5,2} \partial^2  + E_{5,1}\partial+ E_{5,0},
\]
where
\begin{align*}
        E_{5,5} &= {\bf 3y_2' -5u_2'}  ,\\
        E_{5,4} &= {\bf 3y_3' + 3y_2'' - 10u_2'' - 5u_3'} ,\\
        E_{5,3} &= {\bf 3y_4' + 3y_3'' + y_2''' + u_2y_2' -3u_2'y_2 -10u_2''' -10u_3''} ,\\
        E_{5,2} &= {\bf 3y_5'+3y_4''+y_3'''+ u_2y_3' -3u_2''y_2 -2u_2'y_3-3u_3'y_2-5u_2^{(4)} -10u_3'''} ,\\
        E_{5,1} &= 3y_5''+ y_4'''+ u_2y_4'-u_2'''y_2 - u_2''y_3 - u_2'y_4  -3u_3''y_2 -2u_3'y_3 - u_2^{(5)} -5u_3^{(4)},\\
        E_{5,0} &= u_2y_5' - u_3'''y_2 - u_3''y_3 - u_3'y_4 - u_3^{(5)} + y_5'''.
\end{align*}
Solving the triangular system
$\bbE_{5,3}=\{E_{5,5}=0,E_{5,4}=0,E_{5,3}=0,E_{5,2}=0\}$ with
respect to $\{y_2,y_3,y_4,y_5\}$ we obtain $Z=\{q_2,q_3,q_4,q_5\}$
and
\begin{align*}
P_5= \epsilon_Z
 (\tilde{P}_5 ) &=\partial^5 + \frac{5}{3}u_2 \partial^3 + \left(\frac{5}{3}u_2' + \frac{5}{3}u_3\right)\partial^2 + \left(\frac{10}{9}u_2'' + \frac{5}{3}u_3' + \frac{5}{9}u_2^2\right)\partial\\
&\quad{}+ \left(\frac{10}{9}u_3'' + \frac{10}{9}u_2u_3\right).
\end{align*}
The system of the GD hierarchy at level $m=5$ is,
\begin{align*}
    u_{2,t}&=H_{5,1}+c_{5,4} H_{4,1}+c_{5,2} H_{2,1}+c_{5,1} H_{1,1} ,\\ u_{3,t}&=H_{5,0}+c_{5,4} H_{4,0}+c_{5,2} H_{2,0}+c_{5,1} H_{1,0}
\end{align*}
where
\[\left\{\begin{aligned}
    -H_{5,0}(u_2,u_3) & = \frac{10}{9}u_{2}' u_{2} u_{3} + \frac{5}{9}u_{2}^2u_{3}' + \frac{10}{9}u_{2}''' u_{3} + \frac{20}{9}u_{2}'' u_{3}' \\[1em]
                         &\quad {} + \frac{5}{3}u_{2}' u_{3}'' + \frac{5}{9}u_{2} u_{3}''' - \frac{5}{3}u_{3}'' u_{3} - \frac{5}{3}u_{3}'^2 + \frac{1}{9}u_{3}^{(5)},\\[1.5em]
    -H_{5,1}(u_2,u_3) & = \frac{5}{9}u_{2}' u_{2}^2 + \frac{5}{9}u_{2}''' u_{2} + \frac{5}{9}u_{2}'' u_{2}'  + \frac{5}{3}u_{2}'' u_{3} + \frac{5}{3}u_{2}' u_{3}' \\[1em]
                        &\quad{}  - \frac{10}{3}u_{3}' u_{3} + \frac{1}{9}u_{2}^{(5)}.
\end{aligned}\right.\]

More results on the computation of the almost commuting basis and
the differential polynomials defining the GD hierarchy for $n=3$ and
$m\geq 2$ can be found in the repository \texttt{da\_wilson}, see
Section \ref{ExperimentalResults}.

\subsection{{ Fifth} order operators and their GD hierarchy}\label{sec-Fith}


Let us consider the third order operator in $\cR[\partial]$ with
$\cR=\coC\{U\}$, with $U=\{u_2,{ u_3},u_4,u_5\}$ given by
\[L_5 = \partial^5+u_2\partial^3+u_3\partial^2+ u_4\partial + u_5.\]
We show the computation of the homogeneous basis $\cH_9(L_5)$ of the
space of almost commuting operators $W_9(L_5)$ of order less than or
equal to $5$,
\[\cH_9(L_5)=\{P_1, P_2, P_3,P_4,P_5,P_6,P_7,P_8, P_9\}.\]
For each $m\geq 1$, it holds that
\[[P_m , L_5 ] =H_{m,3}(U)\partial^3+H_{m,2}(U)\partial^2+ H_{m,1}(U)\partial + H_{m,0}(U),\]
where  $H_{m,0}(U)$, $H_{m,1}(U)$, $H_{m,2}(U)$ and $H_{m,3}(U)$ are
differential polynomials in $\cR$ that determine all the equations
of the GD hierarchy for $L_5$. An almost commuting operator in
$W_9(L_5)$ has the form
\[A_9=P_9+\sum_{j=1}^8 c_{9,j} P_j.\]

\medskip

The system of the Gelfand-Dickey hierarchy at level $m=9$ in
\eqref{eq-diff-system-LA} is ,
\begin{align*}
    u_{2,t}&=H_{9,3}(U)+\sum_{j=1}^8 c_{9,j} H_{j,3}(U),\\
    u_{3,t}&=H_{9,2}(U)+\sum_{j=1}^8 c_{9,j} H_{j,2}(U),\\
    u_{4,t}&=H_{9,1}(U)+\sum_{j=1}^8 c_{9,j} H_{j,1}(U),\\
    u_{5,t}&=H_{9,0}(U)+\sum_{j=1}^8 c_{9,j} H_{j,0}(U).
\end{align*}

We computed $H_{j,5-i}$ for $1\leq j\leq 9$ and $i=2,\ldots ,5$. The
results of these computations can be found in the repository
\texttt{da\_wilson}, where we have collected several computations of
these hierarchies to provide a  data set that can be used by
researchers in Mathematics and Physics. In particular, we refer to
the folder \texttt{latex}, where the LaTeX representation of these
polynomials $H_{i,j}(U)$ can be found in the files with name
\texttt{(5\_i)[H\_j].tex}. A further description of this repository
can be found in Section~\ref{ExperimentalResults}.

\section{Implementation}\label{ExperimentalResults}
In this paper, we revisit the theory of almost commuting
differential operators and the definition of the homogeneous basis
for these operators. In order to allow future computations, we have
implemented Algorithm~\ref{alg:compute_basis} in the computer
algebra system SageMath~\cite{SAGE2023}, in the package
\texttt{dalgebra}, a set of tools implemented in SageMath by the
first author and dedicated to Differential Algebra. We present in
this section the main features of this implementation, showing  the
time spent for these computations and presenting a public repository
where the results can be retrieved in several formats.

{  Although a complexity analysis of Algorithm
~\ref{alg:compute_basis} was not performed, we would like to point
out that the complexity would depend on: the multiplication of
differential operators in Step 2; and the complexity of the
integration algorithm used in Step 4, the instruction
$\texttt{INTEG}(e_{m,i})$ of Algorithm
\ref{alg:solving_integration}. It is important to emphasize that the
system $\bbE_{m,n}$ of Step 3 is already triangular, as proved in
Theorem \ref{thm-triangularE}, so there is no triangulation
algorithm involved.}

\subsection{Implementation of Algorithm~\ref{alg:compute_basis}}

Algorithm~\ref{alg:compute_basis} has been implemented in the module
\texttt{almost\_commuting.py} inside the SageMath open-source
package for differential algebra \texttt{dalgebra}. This software is
publicly available (in version 0.0.5 when writing this paper) and
can be obtained from the following link:
\begin{center}\url{https://github.com/Antonio-JP/dalgebra/releases/tag/v0.0.5}\end{center}
We refer to the README file in the repository for an installation
guide and examples on how to use the software.

We have included Algorithm~\ref{alg:compute_basis} as the method
$\texttt{almost\_commuting\_wilson}$, which receives two inputs $n$
and $m$ and returns:
\begin{itemize}
    \item The differential polynomial $P_m \in \mathbb{Q}\{U\}[\partial]$ of the basis of almost commuting operators for the operator $L_n = \partial^n + u_{2}\partial^{n-2} + \ldots + u_n$, given by Theorem~\ref{thm-basis}.
    \item The list of $(n-1)$ differential polynomials $H_{m,i}(U) \in \mathbb{Q}\{U\}$ such that
    \[[P_m , L_n ] = H_{m,0}(U) + H_{m,1}(U)\partial + \ldots + H_{m,n-2}(U)\partial^{n-2}.\]
\end{itemize}
For a given pair of values $(n,m)$, this method automatically stores
the result on a serialized file to reuse it whenever is needed.

As a demonstration on the performance of the implemented method, we
show in Figure~\ref{fig:speedup} the time spent for computing the
first almost commuting elements~$P_m$ for $m=2,\ldots,14$ for the
generic operators of orders $n=2,3,5,7$. These time executions have
been performed in a laptop, Intel(R) Core(TM) i7-1165G7 (2.8GHz)
processor and 23GB of RAM memory, under version Sage~9.7.

\begin{figure}
    \centering
    \includegraphics[width=\textwidth]{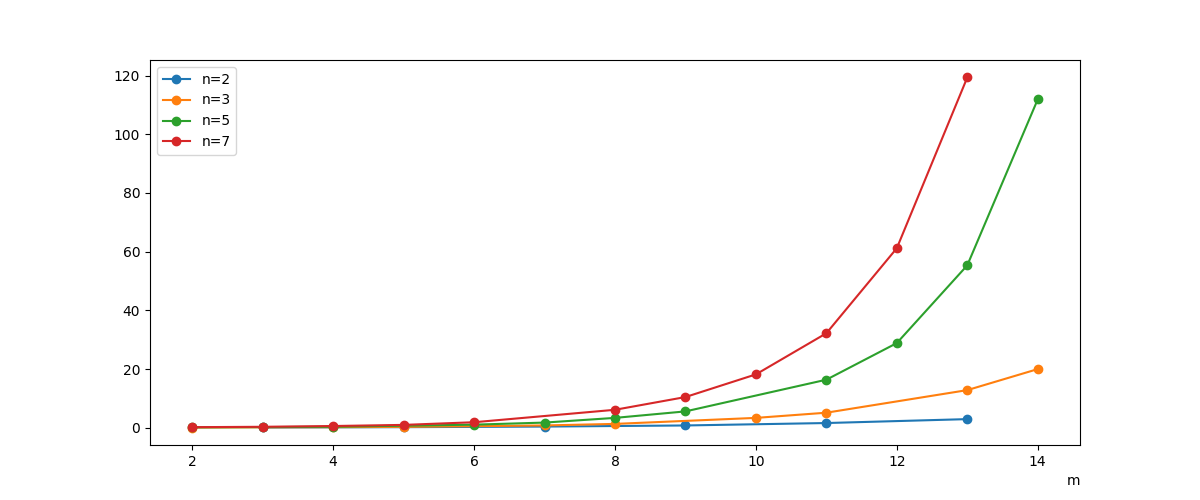}
    \caption{Time (in seconds) spent when calling the method \texttt{almost\_commuting\_wilson} for inputs $(n,m)$ ranging $n=2,3,5,7$ and $m=2,\ldots,14$. Each line represents the different times for a specific value of $n$ while values for $m\equiv 0\ (mod\ n)$ are skipped since they are not representative for time considerations.}
    \label{fig:speedup}
\end{figure}

It can be appreciated that the times are still reasonable (not
bigger than 120 seconds). We would like to remark that the biggest
result (in space) is the case with $n=7$ and $m=13$, where all data
structures take 4MB of disk space. Mathematically speaking, this
result ($n=7$, $m=13$) contains the polynomial~$P_{13}$ for $L_7$
with 830 monomials, and the 6 differential polynomials~$H_{13,i}$
all with degree 7 and ranging from 744 up to 5279 monomials. The
script that generates the figure (i.e., measures the times and
generates the image) can be found in folder \texttt{experiments} in
the aforementioned repository, in the file
\texttt{almost\_commuting.sage}.

\subsection{Results available}

In order to make the computed operators available to public usage,
we have published a GitHub repository with several files with the
results for $n=2,3,5,7$ and $m=2,\ldots,15$. The repository is
called \texttt{da\_wilson} and it is available in the following
link.
\begin{center}
    \url{https://github.com/Antonio-JP/da_wilson}
\end{center}

In this repository we keep for each $n$ and $m$ stored a set of
different files.
\begin{itemize}
    \item An \texttt{.out} file containing the serialize version of the output of method \texttt{almost\_commuting\_wilson}. This contains both the polynomial $P_m$ and the list of $H_{m,i}$.
    \item A series of \texttt{.tex} files containing the LaTeX representation of the elements in the \texttt{.out} file. One file for the polynomial $P_m$ and one file for each of the differential polynomials $H_{m,i}$ for $i=0,\ldots,n-1$.
    \item A series of \texttt{.mpl} files containing a Maple representation of the elements in the \texttt{.out} file. Similar to the \texttt{.tex} files, we store the polynomial $P_m$ and the polynomials $H_{m,i}$ in separated files.
\end{itemize}

The names of the files follow this pattern:
\texttt{(n\_m)[suf].ext}, where;
\begin{itemize}
    \item \texttt{n} and \texttt{m} will take the actual value of the data stored in the file,
    \item \texttt{suf} will indicate what element is in the file. \texttt{P} means the polynomial $P_m$ is in the file, while $\texttt{H\_i}$ indicates that the polynomial $H_{m,i}$ is stored in that file. For the \texttt{.out} files, this suffix will not appear.
    \item \texttt{ext} indicates the extension of the file, meaning the format in which the object is represented.
\end{itemize}

We refer to the \texttt{README} in this repository for advice on how to load data from these files into the different
systems. Also, the script for generating all the Maple and TeX files
is included in this repository.

\medskip

\backmatter

\bmhead{Acknowledgements} All authors are partially supported by the
grant PID2021-124473NB-I00, ``Algorithmic Differential Algebra and
Integrability" (ADAI)  from the Spanish MICINN. R.D. is partially
supported by the Poul Due Jensen Grant 883901.

\bibliography{Bibliography}

\end{document}